\documentclass{article}
\usepackage{amsmath,amssymb,latexsym,amsthm,mathrsfs}

\newtheorem{theorem}{Theorem}[section]
\newtheorem{proposition}[theorem]{Proposition}
\newtheorem{lemma}[theorem]{Lemma}

\newcommand{\ra}{\rightarrow}

\theoremstyle{definition}
\newtheorem{definition}[theorem]{Definition}
\newtheorem{example}[theorem]{Example}
\newtheorem{problem}[theorem]{Problem}

\theoremstyle{remark}
\newtheorem{remark}[theorem]{Remark}

\title{Merging and stability for time inhomogeneous finite Markov chains}
\author{L. Saloff-Coste
        \thanks{Research partially supported by NSF grant DMS 0603886}\\
{\small Department of mathematics}\\
{\small Cornell University}
\and  J. Z\'u\~niga
      \thanks{Research partially supported by NSF grants
DMS 0603886, DMS 0306194, DMS 0803018 and by a NSF Postdoctoral Fellowship.}\\
{\small Department of mathematics}\\
{\small Stanford University}}

\begin{document}
\maketitle

\begin{abstract} We discuss problems posed by the quantitative study of time
inhomogeneous Markov chains. The two main notions for our purpose are 
merging and stability. Merging (also called weak ergodicity)
occurs when the chain asymptotically forgets where it started.
It is a loss of memory property.
Stability relates to the question of whether or not, despite
temporary variations, there is a rough shape describing the long time 
behavior of the chain.
For instance, we will discuss an example where the long time behavior 
is roughly described by a binomial, with temporal variations.
\end{abstract}

\section{Introduction}  As is apparent from most text books, the definition
of a Markov process includes, in the most natural way, processes that are
time inhomogeneous. Nevertheless, most modern references quickly restrict
themselves to the time homogeneous case
by assuming the existence of a time homogeneous transition function, a case
for which there is a vast literature.

The goal of this paper is to point out some interesting problems concerning
the quantitative study of time inhomogeneous 
Markov processes and, in particular,
time inhomogeneous Markov chains on finite state spaces.
Indeed, almost nothing is known about the 
quantitative behavior of time inhomogeneous chains. Even the simplest examples resist analysis. We describe some precise
questions and examples, and a few results. They indicate the extent
of our lack of understanding, illustrate the difficulties and, perhaps,
point to some hope for progress.

We think the problems discussed below have an intrinsic mathematical
interest (indeed, some of them appear quite hard to solve) 
and are very natural. 
Nevertheless,  it is reasonable to ask whether or not time inhomogeneous
chains are relevant in some applications. Most of the recent interest 
in Markov chains is related to
Monte Carlo Markov Chain algorithms. In this context, one seeks
a Markov chain with a given stationary distribution. Hence, 
time homogeneity is  rather natural. See, e.g., \cite{RR}. Still, one of the  popular algorithms
of this sort, the Gibbs sampler, can be viewed as a 
time inhomogeneous chain (one that, despite huge amount of attention,
is still resisting analysis).   Time inhomogeneity also appears  
in the so-called simulated annealing algorithms.  
See \cite{DM} 
for a discussion that is close in spirit to the present work and for older 
references. 
However, certain special 
features of each of these two algorithms distinguish them 
from the more basic time inhomogeneous problems we want to discuss here.
Namely,
in the Gibbs sampler, each individual step is not ergodic 
(it involves only one coodinate) whereas, in the simulated annealing context,
the time inhomogeneity vanishes asymptotically. Other interesting stochastic
algorithms that present time inhomogeneity are discussed in \cite{DG}.

In many  applications of finite Markov chains, the kernel 
describes transitions between different classes in a population of interest. 
Assuming that these 
transition probabilities can be observed empirically, one application  
is to compute the stationary measure which describes 
the steady state of the system.
Examples of this type include models for population migrations between
countries, models for credit scores used to study the default risk of certain
loan portfolios, etc.  In such examples, it is  natural to consider cases when
the Markov kernel describing the evolution of the system depends on time
in either a deterministic or a random manner. 
The reason for the time inhomogeneity
may come, for example,  from seasonal factors. Or it may model various 
external events that are independent of the state of the system.  
Even if one decides that time homogeneity is warranted, 
one may wish to study the possible effects of
small but non-vanishing time dependent perturbations of the model. 
It seems rather important to understand whether or not such perturbations
can drastically alter the behavior of the underlying model.
This type of  practical questions fit nicely with the theoretical 
problems  discussed below.

A large class of natural examples of time inhomogeneous chains comes from
time inhomogeneous random walks on groups. These are discussed in 
\cite{Bo,SZ}. A special case is the semi-random transpositions 
model discussed in \cite{Ga,Mir,MPS,SZ}.

\section{Merging and stability}
This section introduces the two main properties we want to focus on: merging 
(in total variation or relative-sup) and stability. 
Given two Markov kernels  $K_1,K_2$, we set
$$K_1K_2(x,y)=
\sum_zK_1(x,z)K_2(z,y).$$
Given a sequence $(K_i)_1^\infty$ and $0\le m\le n$,
we set  $$K_{m,n}= K_{m+1}\cdots K_n,\;\;  K_{m,m}=I.$$

\subsection{Merging}
Recall that an aperiodic irreducible Markov kernel $K$ on a
finite state space admits a unique invariant probability measure $\pi$.
Further,  for any starting measure $\mu_0$ and any
large time $n$, the distribution $\mu_n=\mu_0K^n$ at time $n$ is
both essentially independent from the starting distribution $\mu_0$ and
well approximated by $\pi$.

Consider now the evolution of a system started according to an initial
distribution $\mu_0$ and driven by a sequence
$(K_i)_1^\infty$ of Markov kernels so that, at time $n$,
the distribution is $\mu_n=\mu_0K_1K_2\cdots K_n$. 
In \cite{Bl,Cohn} such a sequence $(\mu_n)_1^\infty$ of probability measures 
is called a ``set of absolute probabilities'' but we will not use this 
terminology here. In many cases,
for very large $n$, the distribution $\mu_n$ will be essentially
independent of the initial distribution $\mu_0$. Namely, if $\mu_0,\mu'_0$
are two initial distributions and $\mu_n=\mu_0 K_1\cdots K_n$,
$\mu'_n=\mu'_0 K_1\cdots K_n$, then it will often be the case that
$$\lim_{n\ra \infty}\|\mu_n-\mu'_n\|_{\mbox{\tiny TV}}= 0.$$
We call this loss of memory property  {\em \index{merging}merging}
(total variation merging, to be more precise). 

One may also want to know whether or not
$$\lim_{n\ra \infty}\sup_x\left\{\left|\frac{\mu'_n(x)}{\mu_n(x)}-1\right|
\right\}=0.$$
We call this later property relative-sup merging.  
Total variation merging is often discussed under the name of 
``weak ergodicity''. See, e.g.,  \cite{Bl,Cohn,Coh,Gri,Haj,Io,Pa}. 
We think ``merging'' is more appropriate.

If there is merging, then one may want to ask quantitative questions
about the merging time. For any $\epsilon\in (0,1)$, we set
\begin{equation}\label{def-Ttv}
T_1(\epsilon)=\inf \left\{n: \forall \mu_0,\mu'_0, \;
\|\mu_n-\mu'_n\|_{\mbox{\tiny TV}}
\le \epsilon\right\}\end{equation}
and
\begin{equation}\label{def-Trs}
T_\infty(\epsilon)=\inf \left\{n: \forall \mu_0,\mu'_0, \;
\left\|\frac{\mu'_n}{\mu_n}-1\right\|_{\infty}
\le \epsilon\right\}.\end{equation}

The next definition introduces the collective notions of merging 
and merging time for a given set $\mathcal Q$ of Markov kernels.
\begin{definition} Let $\mathcal Q$ be a set of Markov kernels on a finite state 
space. We say that $\mathcal Q$ is merging in total variation 
(resp. relative-sup) if any sequence $(K_i)_1^\infty$ of kernels in 
$\mathcal Q$ is merging in total variation (resp. relative-sup).
We say that $\mathcal Q$ has total-variation (resp. relative-sup)
$\epsilon$-merging time at most $T(\epsilon)$ if the total variation
(resp. relative-sup)
$\epsilon$-merging time (\ref{def-Ttv}) (resp. (\ref{def-Trs}))
is bounded above by $T(\epsilon)$, for any sequence 
$(K_i)_1^\infty$ of kernels in $\mathcal Q$.
\end{definition}

Let us emphasize that, from the view point of the present work, it is more 
natural to think in terms of properties shared by all sequences drawn from 
a set of kernels than   in terms of properties of some particular  sequence.

\subsection{Stability} \label{sec-stab}
In the previous section, the notion of
merging was introduced as a natural generalization  
of the loss of memory property in the time 
inhomogeneous context.
The notion of {\em stability} introduced below is a generalization
of the existence of a positive invariant distribution.

\begin{definition} Fix $c\ge 1$.
Given a Markov chain driven by a sequence of Markov kernels $(K_i)_1^\infty$,
we say  that a probability measure $\pi$ is $c$-stable (for $(K_i)_1^\infty$)
if there exists a positive measure $\mu_0$
such that the sequence $\mu^n=\mu_0K_{0,n}$ satisfies
$$  c^{-1}\pi\le \mu_n\le c\pi.$$
When such a measure $\pi$ exists, we say that $(K_i)_1^\infty$ is 
\index{stable}$c$-stable.
\end{definition}

\begin{example} Let $K$ be an irreducible aperiodic kernel.
Then the chain driven by $K$ is $1$-stable. Indeed, it admits a positive invariant measure $\pi$ and $\pi K^n=\pi$.
Further, for any probability measure $\mu_0$
with $\|(\mu_0/\pi)-1\|_\infty\le \epsilon$, the sequence 
$\mu_n=\mu_0K^n$, $n=1,2,\dots$, satisfies
$ (1-\epsilon)\pi\le \mu_n\le (1+\epsilon)\pi$.
Indeed, in the space of signed measures, the linear map
$\mu\mapsto \mu K$ is a contraction for the distance
$d(\mu,\nu)=\|(\mu/\pi)-(\nu/\pi)\|_\infty$.
\end{example}

In the next definition, we consider the notion of $c$-stability
for a family $\mathcal Q$ of Markov kernels on a fixed state space.
This definition is of interest even in the case when $\mathcal Q=\{Q_1,Q_2\}$
is a pair.

\begin{definition} Fix $c\ge 1$.
Given a set $\mathcal Q$ of Markov kernels on a fixed state space,
we say  that a probability measure $\pi$ is a $c$-stable measure
for $\mathcal Q$
if there exists a positive measure $\mu_0$
such that for any choice of sequence $(K_i)_1^\infty$ in $\mathcal Q$,
the sequence $\mu_n=\mu_0K_{0,n}$ satisfies
$$  c^{-1}\pi\le \mu_n\le c\pi.$$
When such a measure $\pi$ exists, we say that $\mathcal Q$ is $c$-stable.
\end{definition}

\begin{example} Assume the state space is a group $G$ and let $\mathcal Q$
be the set of all Markov kernels $Q$ such that $Q(zx,zy)=Q(x,y)$ for all  $x,y,z\in G$.
This set is $1$-stable with $1$-stable measure $u$, the uniform measure on $G$.
\end{example}

\begin{example} On the two-point space, a finite set $\mathcal Q$
of Markov kernels is $c$-stable if and only if it contains no pairs
$\{Q_1,Q_2\}$ with
$Q_i=\left(\begin{array}{cc}a_i&1-a_i\\ 1-b_i& b_i\end{array}\right)$
such that $Q_1\neq Q_2$, $a_1=0$, $b_2=0$. This condition is clearly necessary.
It is not immediately obvious that it is sufficient.
See \cite{SZ3}.
\end{example}

\begin{remark} Consider the problem of deciding whether or not a pair
$\mathcal Q=\{Q_1,Q_2\}$ of two irreducible ergodic
Markov kernels with
invariant measure $\pi_1,\pi_2$, respectively, is $c$-stable. This can be pictured by
considering a rooted infinite binary tree with edges  labeled
$Q_1$(=left) and $Q_2$(=right)
as on Figure \ref{fig-tree}. Obviously, any sequence $(K_i)_1^\infty$
with $K_i\in \mathcal Q$ corresponds uniquely to an end $\omega\in \Omega$
where $\Omega$ denotes the set of the ends the tree.
Given an initial measure $\mu_0$ (placed at the root), the measure
$\mu^\omega_n=\mu_0K_{0,n}$ is obtained by following  $\omega$ from
the root down to level $n$. Thus, for each choice of $\mu_0$, we obtain a tree with
vertices labeled with measures.

\begin{figure}[h]

\begin{center}
\caption{The $Q_1,Q_2$ tree \label{fig-tree}} \vspace{.05in}
\end{center}

\begin{picture}(300,100)(-30,0)

\put(150,110){\makebox(0,0){$\mu_0$?}}

\put(150,100){\circle*{4}}
\put(150,100){\line(-2,-1){40}}\put(150,100){\line(2,-1){40}}
\put(120,95){\makebox(0,0){$Q_1$}}
\put(175,95){\makebox(0,0){$Q_2$}}
\put(110,80){\circle*{2}} \put(190,80){\circle*{2}}
\put(110,80){\line(-1,-1){20}}\put(110,80){\line(1,-1){20}}
\put(90,72){\makebox(0,0){$Q_1$}}\put(130,72){\makebox(0,0){$Q_2$}}
\put(190,80){\line(-1,-1){20}}\put(190,80){\line(1,-1){20}}
\put(90,60){\circle*{2}} \put(210,60){\circle*{2}}
\put(130,60){\circle*{2}} \put(170,60){\circle*{2}}
\put(90,60){\line(-1,-2){10}}\put(210,60){\line(1,-2){10}}
\put(130,60){\line(-1,-2){10}}\put(170,60){\line(1,-2){10}}
\put(90,60){\line(1,-2){10}}\put(210,60){\line(-1,-2){10}}
\put(130,60){\line(1,-2){10}}\put(170,60){\line(-1,-2){10}}
\put(197,50){\makebox(0,0){$Q_1$}}\put(225,50){\makebox(0,0){$Q_2$}}
\put(75,50){\makebox(0,0){$Q_1$}}
\put(80,40){\circle*{2}} 
\multiput(150,40)(0,-4){5}{\line(0,-1){2}}\put(150,20){\vector(0,-1){2}}
\multiput(70,0)(4,0){41}{\line(1,0){2}}
\put(148,3){\makebox{$\Omega$}}\put(70,3){\makebox{$\pi_1$}}
\put(226,3){\makebox{$\pi_2$}}
\end{picture}

\end{figure}
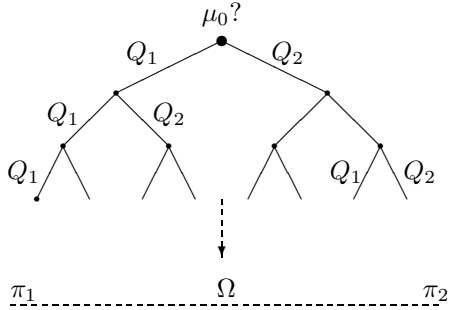
The question of $c$-stability is the problem of finding an
initial measure $\mu_0$ which, in some sense, minimizes the variations
among the $\mu^\omega_n$'s. At the left-most and right-most ends $\omega_1$,
$\omega_2$, we get $\mu^{\omega_i}_n\ra \pi_i$. Note that, if $Q_1,Q_2$
share the same invariant measure $\pi_1=\pi_2=\pi$, then the choice $\mu_0=\pi$
yields a tree all of whose vertices are labeled by $\pi$. The existence of
a $c$-stable measure $\mu_0$ can be viewed as a weakening of this.
The difficulty is that the existence of an invariant measure and
thus the equality between $\pi_1$ and $\pi_2$ can be viewed as
an algebraic property whereas there seems to be no algebraic tools to study
$c$-stability.
\end{remark}

\subsection{Simple results and examples} \label{FSS}
We are interested in finding  conditions on the individual kernels
$K_i$ of a sequence $(K_n)_1^\infty$ that imply merging. 
This is not obvious even if we consider
the very special case when all the $K_i$'s are drawn from a
finite set of kernels
$\mathcal Q=\{Q_0,\dots,Q_m\}$ or even from a pair 
$\mathcal Q=\{Q_0,Q_1\}$.
\begin{itemize}
\item Suppose that $Q_0,Q_1$ are
irreducible and aperiodic. Does it imply any sequence $(K_i)_1^{\infty}$
drawn from $\mathcal Q=\{Q_0,Q_1\}$ is merging?
\end{itemize}
The answer is no.
Let $\pi_0$ be the invariant measure of $Q_0$ and let
$Q_1=Q^*_0$  be the adjoint of $Q_0$ on $\ell^2(\pi_0)$. If $(Q_0,\pi_0)$
is not reversible (i.e., $Q_0$ is not self-adjoint on $\ell^2(\pi_0)$)
then it is possible that $Q_0Q_0^*$ is not irreducible. When $Q_0Q_0^*$
is not irreducible, the sequence  $K_i=Q_{i\mod 2}$ is not merging.
\begin{itemize}
\item Suppose that $Q_0,Q_1$ are reversible,
irreducible and aperiodic. Does it imply any sequence $(K_i)_1^{\infty}$
drawn from $\mathcal Q=\{Q_0,Q_1\}$ is merging in relative-sup?
\end{itemize}
The answer is no, even on the two point space!
On the two point space, $\mathcal Q=\{Q_0,Q_1\}$ is merging in total
variation as long as $Q_0,Q_1$ are irreducible aperiodic but relative
sup merging fails for the irreducible aperiodic pairs of the type
$$Q_0=\left(\begin{array}{cc} 0&1\\ 1-a&a\end{array} \right),
Q_1= \left(\begin{array}{cc} b&1-b\\ 1&0\end{array} \right),$$
with $0<a,b<1$. See \cite{SZ3}.

The following examples are instructive. 
\begin{example}
On $S=\{1,\dots,5\}$
consider the reversible kernels $Q_0,Q_1$
corresponding to the graphs in Figure \ref{fig0} (all edges have weight $1$).
Consider the sequence $K_i=Q_{i\mod 2}$ so that
$K_1=Q_1, K_2=Q_0, K_3=Q_1,\dots$.
If, at an even time $n=2\ell$,
the chain is at states $2$ or $5$ then from that time on, the chain will
be in $\{2,5\}$ at even times and in $\{3,4\}$ at odd times.  In this example,
the chain driven by $(K_i)_1^\infty$ is merging in total variation
but is not merging in relative-sup.

\begin{figure}[h]

\begin{center}
\caption{A five-point example} 
\vspace{.05in}  
\label{fig0}
\end{center}

\begin{picture}(300, 50)(0,0)
\put(60,40){\makebox{$Q_0$}}
\put(50,25){\circle*{3}} \put(42,25){\circle{15}}
\put(49,14){\makebox{$1$}}\put(52,25){\line(1,0){20}}
\put(75,25){\circle*{3}}\put(72,16){\makebox{$2$}}
\put(100,50){\circle*{3}} \put(96,53){\makebox{$3$}}
\put(77,27){\line(1,1){20}}
\put(100,0){\circle*{3}} \put(97,-9){\makebox{$4$}}
\put(77,23){\line(1,-1){20}}
\put(125,25){\circle*{3}}\put(123,15){\makebox{$5$}}
\put(123,27){\line(-1,1){20}}\put(123,23){\line(-1,-1){20}}

\put(150,0){\put(60,40){\makebox{$Q_1$}}
\put(50,25){\circle*{3}} \put(42,25){\circle{15}}
\put(49,14){\makebox{$1$}}\put(52,25){\line(1,0){20}}
\put(75,25){\circle*{3}}\put(72,16){\makebox{$3$}}
\put(100,50){\circle*{3}} \put(96,53){\makebox{$2$}}
\put(77,27){\line(1,1){20}}
\put(100,0){\circle*{3}} \put(97,-9){\makebox{$5$}}
\put(77,23){\line(1,-1){20}}
\put(125,25){\circle*{3}}\put(123,15){\makebox{$4$}}
\put(123,27){\line(-1,1){20}}\put(123,23){\line(-1,-1){20}}
}
\end{picture}
\end{figure}
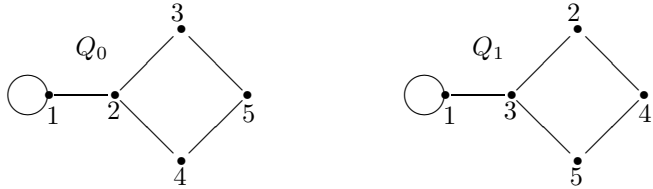
\end{example}

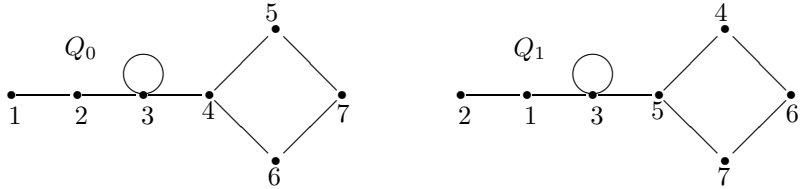
\begin{figure}[h]

\begin{center}
\caption{A seven-point example} 
\vspace{.05in}\label{fig1}
\end{center}

\begin{picture}(300, 50)(-10,0)
\put(20,40){\makebox{$Q_0$}}
\put(0,25){\circle*{3}}
\put(-1,14){\makebox{$1$}}
\put(2,25){\line(1,0){20}}
\put(25,25){\circle*{3}}
\put(24,14){\makebox{$2$}}
\put(27,25){\line(1,0){20}}
\put(50,25){\circle*{3}} \put(50,33){\circle{15}}
\put(49,14){\makebox{$3$}}
\put(52,25){\line(1,0){20}}
\put(75,25){\circle*{3}}\put(72,16){\makebox{$4$}}
\put(100,50){\circle*{3}} \put(96,53){\makebox{$5$}}
\put(77,27){\line(1,1){20}}
\put(100,0){\circle*{3}} \put(97,-9){\makebox{$6$}}
\put(77,23){\line(1,-1){20}}
\put(125,25){\circle*{3}}\put(123,15){\makebox{$7$}}
\put(123,27){\line(-1,1){20}}\put(123,23){\line(-1,-1){20}}

\put(170,0){ \put(20,40){\makebox{$Q_1$}}
\put(0,25){\circle*{3}}
\put(-1,14){\makebox{$2$}}
\put(2,25){\line(1,0){20}}
\put(25,25){\circle*{3}}
\put(24,14){\makebox{$1$}}
\put(27,25){\line(1,0){20}}
\put(50,25){\circle*{3}} \put(50,33){\circle{15}}
\put(49,14){\makebox{$3$}}
\put(52,25){\line(1,0){20}}
\put(75,25){\circle*{3}}\put(72,16){\makebox{$5$}}
\put(100,50){\circle*{3}} \put(96,53){\makebox{$4$}}
\put(77,27){\line(1,1){20}}
\put(100,0){\circle*{3}} \put(97,-9){\makebox{$7$}}
\put(77,23){\line(1,-1){20}}
\put(125,25){\circle*{3}}\put(123,15){\makebox{$6$}}
\put(123,27){\line(-1,1){20}}\put(123,23){\line(-1,-1){20}}}
\end{picture}
\end{figure}

\begin{example}\label{ex-mod}
The kernels depicted  in Figure \ref{fig1} 
yield an example where  total variation
(hence, a fortiori, relative-sup) merging fails.
In this example, the sequence $(K_i)_1^\infty$ with $K_i=Q_{i\mod 2}$
fails to be merging in total variation because the chain will eventually
end up  oscillating
either between $2$ and $1$, or between  $\{4,7\}$ and $\{5,6\}$, with a
preference for one or the other depending on the starting distribution $\mu_0$.
\end{example}

Let us give two simple results concerning merging.
\begin{proposition} \label{pro1}
Assume that, for each $i$,  there exists a state $y_i$ and a real
$\epsilon_i\in (0,1)$ such that
$$\forall\, x,\;\;K_i(x,y_i)\ge \epsilon_i.$$
If $\sum_i\epsilon_i=\infty$ then the sequence $(K_i)_1^\infty$ is merging
in total variation. If, in addition, each $K_i$ is irreducible
then the sequence $(K_i)_1^\infty$ is also merging
 in relative-sup.
\end{proposition}

 \begin{proof} For total variation, this can be proved by a
well-known Doeblin's coupling argument (see, e.g., \cite{DMR,SZ3})
and irreducibility of the kernels is not needed. Of course, the
mass might
ultimately  concentrate 
on a fraction of the state space.

Merging in relative-sup is a bit more subtle and irreducibility is needed
for that conclusion to hold (even in the time homogeneous case).
A proof using singular values can be found in \cite{SZ3}.
\end{proof}

\begin{remark}Under the much stronger hypothesis
$\forall\, x,y,\;\;K_i(x,y)\ge \epsilon_i>0 $, one gets an immediate
control of any sequence $\mu_n=\mu_0K_{0,n}$, $n=1,2,\dots$, in the form
$$\forall\,z,\;\; \epsilon_n\le \min_{x,y}\{K_n(x,y)\}\le \mu_n(z)\le
\sup_{x,y}\{K_n(x,y)\}\le 1-(N-1)\epsilon_n$$
where $N$ is the size of the state space.
\end{remark}
\begin{remark} \label{rem-block}
The hypothesis $\exists\,y_i,\forall\, x,\;\;K_i(x,y_i)\ge \epsilon_i>0$,
is obviously too strong in many cases but it can often be applied
to study a time inhomogeneous chain $(K_i)_1^\infty$ by grouping terms
and considering the sequence $Q_i=K_{n_i,n_{i+1}}$ for an  appropriately
chosen increasing sequence  $n_i$. In the simplest case,
for a given sequence $(K_i)_1^\infty$, one seeks $\epsilon\in(0,1)$ and
an integer $m$ such that
$K_{\ell m, \ell m+m}(x,y)\ge \epsilon$ for all $x,y, \ell$. When such
a lower bound holds, one concludes that
(1) the chain is merging in total variation and relative-sup and (2)
there exists $c\in(0,1)$ such that for any starting measure $\mu_0$
and $n$ large enough, the measures $\mu_n=\mu_0K_{0,n}$ satisfy
$c\le \mu_n(z) \le  1-c$. However, this type of argument is bound to yield 
very  poor quantitative results in most cases.
\end{remark}

For the next result, recall that an adjacency matrix $A$ is a matrix whose entries are either $0$ or $1$.
\begin{proposition} \label{pro2}
On a finite state space let $(K_i)_1^\infty$ be a
sequence of Markov kernels.  Assume that:
\begin{enumerate}
\item  (Uniform irreducibility)
 There exist  an $\ell$, $\epsilon\in (0,1)$ and
adjacency matrices
$(A_i)_1^\infty$, such that,
$\forall\,i,x,y,\;\;   A^\ell_i(x,y)>0 \mbox{ and }
K_i(x,y) \ge \epsilon A_i(x,y) .$
\item (Uniform laziness) There exists $\eta\in (0,1)$ such that,
$\forall\, i,x$, $K_i(x,x)\ge \eta$.
\end{enumerate}
Then the chain driven by $(K_i)_1^\infty$ is merging in total variation
and relative-sup norm. Moreover, there exists $n_0$ and $c\in(0,1)$ such
that for any starting distribution $\mu_0$, all $n\ge n_0$ and all $z$,
$\mu_n=\mu_0K_{0,n}$  satisfies $\mu_n(z)\in (c,1-c)$.
\end{proposition}
\begin{proof} Let $N$ be the size of the state space.
Using (1)-(2), one can show (see \cite{SZ3}) that
$K_{n,n+N}(x,y)\ge  (\min\{\epsilon,\eta\})^{N-1}$. The desired result
follows from Proposition \ref{pro1} and Remark \ref{rem-block}.
\end{proof}
Note that this argument  can only give very poor quantitative results!

\section{A short review of the literature}

The largest body of literature concerning time inhomogeneous Markov processes
come, perhaps, from the analysis of Patial Differential Equations 
where time dependent coefficients are allowed. The book \cite{SV} can serve as a
basic reference. Unfortunately, it seems that 
the results developed in that context are local in nature and 
are not very relevent to the
quantitative problems we are interested in. The literature on (finite) 
time inhomogeneous Markov chains can be organized under three basic headings:
Weak ergodicity, asymptotic structure, and products of stochastic matrices.
We now briefly review each of these directions.

\subsection{Weak ergodicity} One of the earliest references concerning 
the asymptotic  behavior of time inhomogeneous chains is a note of Emile Borel
\cite{Bo} where he discusses time inhomogeneous card shufflings. 
In the context of general time inhomogeneous chains on finite 
state spaces, {\em weak ergodicity}, which we call {\em
total variation merging}, i.e., the tendency to forget the distant past,
was introduced in \cite{Kol} and 
is the main subject of \cite{Haj}. See also \cite{Cohn2} and 
the reference to the work of Doeblin given there.
A sample of additional old and not so old references 
in this direction is \cite{Gri,Io,NS,Pa,Rho,Sen}.  An historical review 
is given in \cite{Sen2}. The main tools developed in these references
to prove weak ergodicity are  the use of ergodic coefficients and couplings. 
A modern perspective, close in spirit to our interests, is in \cite{DG,DLM,DMR}.
It may be worth pointing out that, by design, ergodic coefficients mostly 
capture some asymptotic properties and are not well 
suited for quantitative results, even in the time homogeneous case.

\subsection{Asymptotic structure} One of the basic results in the theory of 
time homogeneous finite Markov chains describes the decomposition of the state 
space into non-essential (or transient) states, essential classes and 
periodic subclasses. It turns out that, perhaps surprisingly,
there exists a completely general  version of this result for time 
inhomogeneous chains. This result is rather more subtle than its
time homogeneous counterpart. Sonin \cite[Theorem 1]{Son} calls it the
\index{Decomposition-Separation Theorem}Decomposition-Separation Theorem 
and  reviews its history
which starts with a paper of Kolmogorov \cite{Kol}, with further important 
contributions by Blackwell \cite{Bl}, Cohn \cite{Cohn} and Sonin \cite{Son}.

Fix a sequence $(K_n)_1^\infty$ of Markov kernels on a finite state space $\Omega$.
The Decomposition-Separation Theorem yields a sequence 
$(\{S^k_n,k=0,\dots c\})_{n=1}^\infty$ of partitions of $\Omega$
so that: (a) With probability one, the trajectories of any 
Markov chain $(X_n)$ driven by $(K_n)_1^\infty$ will, after a 
finite number of steps, enter one of the sequence $S^k=(S^k_n)_{n=1}^\infty$, 
$k=1,\dots ,c$, and stay there forever. Further, for each $k$, 
$$\sum_{n=1}^\infty \mathbf P(X_n\in S^k_n; X_{n+1}\notin S^k_{n+1})+
\mathbf P(X_n\notin S^k_n; X_{n+1}\in S^k_{n+1})<\infty.$$

(b) For each $k=1,\dots,c$, and 
for any two Markov chains $(X^1_n)_1^\infty,(X^2_n)_1^\infty$ driven by 
$(K_n)_1^\infty$ such that
$\lim_{n\ra \infty}\mathbf P(X^i_n\in S^k_n)>0$, and any sequence 
of states $x_n\in S^k_n$,
$$\lim_{n\ra \infty} 
\frac{\mathbf P(X^1_n=x_n| X_n^1\in S^k_n)}{\mathbf P(X^2_n=x_n| X_n^2\in S^k_n)}=1.$$

The sequence $(S^0_n)_1^\infty$ describes ``non-essential states'' and 
a chain is weakly ergodic (i.e., merging in total variation) 
if and only if $c=1$, i.e., there is only one essential class.
We refer the reader to \cite{Son} for a detailled discussion and 
connections with other problems.  

The Decomposition-Separation Theorem
can be illustrated  (albeit, in a rather trivial way) 
using Example \ref{ex-mod} of Figure \ref{fig1} above. 
In this case, $\Omega=\{1,\dots,7\}$. We 
consider the sequence of partitions $(S_n^k)$, $k\in\{0,1,2\}$, where
$S_{2n}^0=\{1,3,5,6\}$, $S_{2n+1}^0=\{2,3,4,7\}$, $S_{2n}^1=\{2\}$, 
$S_{2n+1}^1=\{1\}$ and  $S_{2n}^2=\{4,7\}$,  $S_{2n+1}^2=\{5,6\}$.
Any chain driven by $Q_1,Q_0,Q_1,\dots$ will eventually end up  
staying either in   $S^1_n$ or in $S^2_n$ forever.

The Decomposition-Separation Theorem is a very general result which 
holds without any hypothesis on the kernels $K_n$. We are instead 
interested in finding  hypotheses, perhaps very restrictive ones,
on the individual kernels $K_n$ that translate
into strong quantitative results concerning the merging property of the chain.

\subsection{Products of stochastic matrices}
There is a rather rich literature on the study 
of products of stochastic matrices. Recall that stochastic matrices 
are matrices with non-negative entries and  row sums equal to $1$. 
This last assumption, which breaks the row/column symmetry,  implies that 
there is significant differences between forward and backward products 
of stochastic matrices.  Given a sequence $K_i$ of stochastic matrices
The forward products form the sequence 
$$K^f_{0,n}=K_1K_2\cdots K_n,\;\; n=1,\dots,$$
whereas the backward products  form the sequence  
$$K^b_{0,n}= K_n\cdots K_2K_1,\;\;
n=1,\dots.$$   There is a crucial difference between these two sequences:
The entries $K^f_{0,n}(x,y)$ do not have any general monotonicity properties
but, for any $y$,  
$$ n\mapsto M(n,y)=\max _x\{K^b_{0,n}(x,y)\} $$
is monotone non-increasing and 
$$n\mapsto m(n,y)=\min _x\{K^b_{0,n}(x,y)\} $$
is monotone non-decreasing. These properties 
are obvious consequences of the fact that the matrices $K_i$ 
are stochastic matrices.  Of course, 
$\lim_{n\ra \infty} M(n,y)$ and $\lim_{n\ra \infty} m(n,y)$ exist for all $y$.

If, for some reason, we know that 
$$\forall\, x,x',\;\;\lim_{n\ra \infty}\sum_y| K^b_{0,n}(x,y)-K^b_{0,n}(x',y)|=0$$
then it follows that the backward products converge to a row-constant matrix 
$\Pi$, i.e.,
$$\forall\,x,x',y,\;\;\Pi(x,y)=\lim_{n\ra \infty} 
K^b_{0,n}(x,y),\;\;\Pi(x,y)=\Pi(x',y).$$
The references \cite{Haj,Kol,NS,Rho,Ste,Wol} form
a sample of old and recent works dealing with this observation.

Changing viewpoint and notation somewhat, consider all finite products of 
matrices drawn from a set $\mathcal Q$ of $N\times N$ stochastic matrices. 
For $\omega=(\dots,K_{i-1},K_i,K_{i+1},\dots)\in \mathcal Q^\mathbb Z$ 
a doubly infinite sequence of matrices and $m\le n\in \mathbb Z$,  set
$$K^\omega_{m,n}= K_{m+1}\cdots K_n , \;\;(K_{m,m}=I).$$

A stochastic matrix is called (SIA)
if its products converge to a constant row matrix.
Here, (SIA) stands for  stochastic, irreducible and  aperiodic 
although ``irreducible'' really means that the matrix has a unique 
recurrent class (transient states are allowed so that the constant row 
limit matrix may have some $0$ columns). 
A central result in this area (e.g., \cite{Ste,Wol}) is that, 
if  $\mathcal Q$ is finite and all finite 
products of matrices in $\mathcal Q$ are (SIA) then, for any  
doubly infinite sequence $\omega\in \mathcal Q^\mathbb Z$,
\begin{equation}\label{W1}
\lim_{n-m\ra \infty}\sum_y|K^\omega_{m,n}(x,y)-K^\omega_{m,n}(x',y)|=0
\end{equation}
and
\begin{equation}\label{W2}
\lim_{m\ra -\infty} K^\omega_{m,n}=\Pi^\omega_n
\end{equation}
where $\Pi^\omega_n$ is a row-constant  matrix. Let $\pi^\omega_n$
be the probability measure corresponding to the rows of row-constant matrix $\Pi^\omega_n$. 
Observe that (\ref{W1})-(\ref{W2}) imply
$$\lim_{n\ra \infty}\sum_y|K^\omega_{0,n}(x,y)-\pi^\omega_n(y)|=0.$$

The following proposition establishes some relations between these
considerations, total variation merging and stability.
\begin{proposition} Let $\mathcal Q$ be a set of $N\times N$ stochastic matrices.
Assume that $\mathcal Q$ is merging (in total variation) and
$c$-stable w.r.t.\ a positive measure $\pi$.  Then 
\begin{enumerate}
\item Any finite product $P$ of matrices in $\mathcal Q$ is 
irreducible aperiodic and its unique positive invariant measure $\pi_P$
 satisfies  $c^{-1}\pi\le \pi_P\le c\pi$.
\item  For any $\omega\in \mathcal Q^\mathbb Z$ and any $n\in \mathbb Z$, 
$\pi^\omega_n$ satisfies $c^{-1}\pi\le \pi^\omega_n\le c\pi$, i.e., 
any limit row $\pi'$ of backward products
of matrices in $\mathcal Q$ satisfies $c^{-1}\pi\le \pi'\le c\pi$. 
\end{enumerate}
\end{proposition}
\begin{proof} (1) As $\mathcal Q$ is $c$-stable w.r.t.\ $\pi$, 
there exists a positive measure $\mu_0$ such that for any finite product $P$ of matrices in $\mathcal Q$ and any $n$, $c^{-1}\pi\le \mu_0P^n \le c\pi$. Since $\mathcal Q$ 
is merging, we must have $\lim_{n\ra \infty} P^n =\Pi_P$ with $\Pi_P$ 
having constant rows, call them $\pi_P$. This implies
$c^{-1}\pi\le \pi_P\le c \pi$. Since $\pi$ is positive, 
$\pi_P$ must be positive and  $\lim_{n\ra \infty} P^n =\Pi_P$ implies 
that $P$ is irreducible aperiodic. We note that (1) is, in fact, 
a sufficient condition for stability. See \cite[Prop. 4.9]{SZ3}.
Under the hypothesis that $\mathcal Q$ is merging, (1) is thus 
a necessary and sufficient condition for $c$-stability.

(2)  Fix $\omega\in \mathcal Q^\mathbb Z$. By hypothesis, on the one hand, 
there exists a positive probability measure $\mu_0$ such that
$ c^{-1} \pi\le \mu_0K^\omega_{m,n}\le c\pi$. On the other hand, merging imply
that  $\lim_{m\ra -\infty} K^\omega_{m,n}=\Pi^\omega_n$ and thus,
$\lim_{m\ra -\infty}\mu_0 K^\omega_{m,n}=\pi^\omega_n$. The desired result 
follows. \end{proof}

\subsection{Product of random stochastic matrices}
For pointers to the literature on products of random stochastic matrices
and Markov chains in a random environment, see, e.g., 
\cite{Cog,Coh,Ros,Tak} and the references therein. 
We end this section 
with short comments regarding the simplest case of products of random 
stochastic matrices, i.e.,  the case where the matrices $K_i$
form an i.i.d sequence of stochastic matrices. 
The backward and forward products
$K^b_{0,n}=K_n\cdots K_1$, $K^f_{0,n}=K_1\cdots K_n$ become random 
variables taking values in the set of all $N\times N$ stochastic matrices.
Although these two sequences of random variables have very different 
behavior as $n$ varies, $K^b_{0,n}$ and 
$K^f_{0,n}$ have the same law.  Takahashi \cite{Tak}
proves that if 
$$\forall\,x,x',\;\;\lim_{n\ra \infty} \sum_y|K^f_{0,n}(x,y)-K^f_{0,n}(x',y)|=0
\;\mbox{ almost surely}$$
then  $K^f_{0,n}$ converges in law and the limit law is that of 
the limit random variable $\lim_{n\ra \infty}K^b_{0,n}$.
Rosenblatt \cite{Ros} applies 
the theory of random walks on semigroups to show that
the Cesaro sums  $n^{-1}\sum_1^nK^f_{0,j}(x,y)$ always converge to a constant
 almost surely.  The articles \cite{Cog,Coh} discuss similar results 
under more general hypotheses on the nature of the random sequence 
$(K_i)_1^\infty$. Unfortunately, these interesting results concerning random 
environments do not shed much light on the quantitative questions 
emphasized here.

\section{Quantitative results and examples}

Informally, the question we want to focus on is the following. Let $(K,\pi)$
be an irreducible aperiodic  Markov kernel and its stationary probability
measure. Let $(K_i)_1^\infty$ be a sequence of Markov kernels so that,
for each $i$,  $K_i$ is a perturbation of $K$ with invariant measure $\pi_i$
that is a perturbation of $\pi$ (what ``perturbation'' means here
is left open on purpose). For an initial distribution $\mu_0$,
consider  the associated sequence of measures defined by
$\mu_n=\mu_0K_1\cdots K_n$, $n=1,2,\dots$.

\begin{problem} (1) Does total variation merging hold?

(2) Does relative-sup merging hold?

(3)  Does there exists $c\ge 1$ such that, for $n$ large enough,
 $$\forall\,x,\;\;\;c^{-1}  \le \frac{\mu_n(x)}{\pi(x)}\le c?$$
\end{problem}
Obviously, these questions call for quantitative results describing
the merging times,  the constant $c$ and the ``large'' time $n$ 
in terms of bounds on the allowed perturbations.

To understand what is meant by quantitative results,
it is easier to consider a family of
problems depending on a parameter representing the size and  complexity
of the problem. So, one starts with a family $(\Omega_N,K_N,\pi_N)$
of ergodic Markov kernels depending on the parameter $N$ whose mixing time
sequence   $(T_1(N,\epsilon))_1^\infty$ (say, in total variation)
is understood.
Then, for each $N$, we consider perturbations $(K_{N,i})_{i=1}^\infty$
of $K_N$ with stationary measure $\pi_{N,i}$ close to $\pi_N$ and ask if the
merging time of $(K_{N,i})_{i=1}^\infty$ can be controlled in terms of
$T_1(N,\epsilon)$.

\begin{problem} \label{Pb0}
Let $\Omega_N=\{0,\dots,N\}$.
Let $\mathcal Q_N$ be the set of all
birth and death chains  $Q$ on $V_N$ with
$Q(x,x+\epsilon)\in [1/4,3/4]$ for all $x,x+\epsilon\in V_N$,
$\epsilon\in \{-1,0,1\}$
and with reversible measure $\pi$ satisfying $1/4\le (N+1)\pi(x)\le 4$,
$x\in V_N$.
\begin{enumerate}
\item Prove or disprove that there exists
a constant $A$ independent of $N$ such that
$\mathcal Q_N$ has total variation $\epsilon$-merging time
at most $AN^2(1+\log_+ 1/\epsilon)$.
\item
Prove or disprove that there exists
a constant $A$ independent of $N$ such that
$\mathcal Q_N$ has relative-sup $\epsilon$-merging time  at most
$AN^2(1+\log_+ 1/\epsilon )$.
\item Prove or disprove that there exist constants $A,C\ge 1$,
such that, for any $N$ and any sequence $(K_i)_1^\infty
\in \mathcal Q_N$, we  have
$$\forall\,x,y\in \Omega_N,\;\;
\forall\,n\ge AN^2,\;\;   \frac{1}{C(N+1)}\le K_{0,n}(x,y)\le \frac{C}{N+1}.$$
\end{enumerate}
\end{problem}

Here the time homogeneous  model is the birth and death chain
$K_N$ with constant rates $p=q=r=1/3$ and $\pi_N=1/(N+1)$, so that
$K_N(x,y)=0$ unless $|x-y|\le 1$, $K(0,0)=K(N,N)=2/3$ and
$K(x,x)=K(x,x\pm 1)=1/3$ otherwise. Of course, it is well known that
$T_1(K_N,\epsilon)\simeq T_\infty(K_N,\epsilon)\simeq N^2(1+\log_+ (1/\epsilon))$
for small $\epsilon>0$. Problem 1.2 asks whether or not these
mixing/merging times are stable under suitable time inhomogeneous
perturbations of $K_N$ and whether or not the limiting behavior stays
comparable to that of the model chain.
To the best of our knowledge
the answer is not known and this innocent looking problem should
be taken seriously.

There appears to be only a small number of papers that attempt 
to prove quantitative results for time inhomogeneous chains.
These include \cite{DLM,DMR,Ga,Mir,MPS} and the authors' works
\cite{SZ, SZ3,SZ4,SZwave}. 
The works \cite{Ga,Mir,MPS,SZ} treat 
only examples of time inhomogeneous chains 
that admit an invariant measure. Technically, this is a very 
specific hypothesis and, indeed, these works show that many of the well 
developed techniques that have been used to study time homogeneous 
chains can be successfully applied under this hypothesis.

\subsection{Singular values}

A typical qualitative result about finite Markov chains is that
an irreducible aperiodic chain is ergodic. We do not know  of
any quantitative versions of this statement. Let $K$ be an irreducible
aperiodic  Markov kernel with stationary measure $\pi$ so that
$\mu_n=\mu_0 K^n\ra \pi$ as $n$ tends to infinity, for any starting
distribution $\mu_0$.

If $(K,\pi)$
is reversible (i.e., $\pi(x)K(x,y)=\pi(y)K(y,x)$) and if $\beta$
denotes the second largest absolute value of the
eigenvalues of $K$ acting on $\ell^2(\pi)$ then $\beta<1$ and
\begin{equation}\label{quant}
2\|\mu_n-\pi\|_{\mbox{\tiny TV}}\le \|\mu_0/\pi\|_2 \beta^n
\end{equation}
where $\|\mu_0/\pi\|_2$ is the norm of
$f_0=\mu_0/\pi$ in $\ell^2(\pi)$.
This can be considered as a quantitative result
although it involves the perhaps unknown reversible measure $\pi$.

If $(K,\pi)$ is not reversible, the inequality still holds with $\beta$
being the second largest singular value of $K$ on $\ell^2(\pi)$
(i.e., the square root of the second largest eigenvalue of $KK^*$
where $K^*$ is the adjoint of $K$ on $\ell^2(\pi)$). However,
it is then possible that $\beta=1$,  in which case the inequality fails to
capture the qualitative ergodicity of the chain.

Inequality (\ref{quant}) has an elegant generalization to the
time inhomogeneous setting. Let $(K_i)_1^\infty$ be a sequence of irreducible
Markov kernels (on a finite state space).
Fix a positive probability measure $\mu_0$
(by positive we mean here that $\mu_0(x)>0$ for all $x$) and set
$$\mu_n=\mu_0K_{0,n}.$$
In the time inhomogeneous setting, we want to compare this sequence of measures
$(\mu_n)_1^\infty$ to the sequence of measures
$(K_{0,n}(x,\cdot))_1^\infty$ describing the distribution 
at time $n$ of the chain started at an arbitrary point $x$.

To state the result, for each $i$,
consider $K_i$ as a linear operator acting from $\ell^2(\mu_i)$
to $\ell^2(\mu_{i-1})$. One easily checks that this operator is a contraction.
Its \index{singular value}singular values 
are the square roots of the eigenvalues of the operator
$P_i=K_i^*K_i: \ell^2(\mu_{i})\ra \ell^2(\mu_{i})$ where $K_i^*:
\ell^2(\mu_{i-1})\ra \ell^2(\mu_i)$ is the adjoint operator which is 
a Markov operator with
kernel
$$K^*_i(x,y)=  \frac{K_i(y,x)\mu_{i-1}(y)}{\mu_{i}(x)}.$$
We let
$$\sigma_i=\sigma(K_i,\mu_i,\mu_{i-1})$$ be the second largest
singular value of $K_i:\ell^2(\mu_i)\ra \ell^2(\mu_{i-1})$. It is the
square root
of the second
largest eigenvalue of the Markov kernel
\begin{equation}\label{Pi}
 P_i(x,y)= \frac{1}{\mu_{i}(x)}\sum_z  K_i(z,x)K_i(z,y)\mu_{i-1}(z).
\end{equation}

\begin{theorem} \label{th-sing}
With the notation introduced above, we have
$$\|K_{0,n}(x,\cdot)-\mu_n\|_{\mbox{\em \tiny TV}}\le \mu_0(x)^{-1/2}
\prod_1^n \sigma_i$$
and
$$\left|\frac{K_{0,n}(x,y)}{\mu_n(y)}-1\right|\le [\mu_0(x)\mu_n(y)]^{-1/2}
\prod_1^n \sigma_i$$
\end{theorem}
For the proof, see \cite{DLM,SZ3}. 
The proofs given in \cite{DLM} and \cite{SZ3}
are rather different in spirit, with \cite{DLM} avoiding the explicit use
of singular values. Introducing singular values allows for further refinements
and is useful for practical estimates.
See \cite{SZ,SZ3}. When coupled with the hypothesis of $c$-stability, 
the above result becomes a powerful and very applicable tool. See, e.g.,
 \cite[Theorem 4.11]{SZ3} and the examples treated in \cite{SZ3,SZ4}.
Unfortunately, proving $c$-stability is not an easy task.

A good example of application of Theorem \ref{th-sing} is the 
following  result taken from \cite{SZ3}. We refer the reader to \cite{SZ3} for 
the proof.
\begin{theorem} Fix $1<a<A<\infty$.
Let $\mathcal Q_N(a,A)$ be the set of all constant  
rate birth an death chains on $\{0,\dots,N\}$
with parameters $p,q,r$ satisfying $p/q\in [a,A]$. 
The set $\mathcal Q_N(a,A)$ is merging in relative-sup 
with relative-sup $\epsilon$-merging time bounded above by 
$$T_\infty(\epsilon)\le C(a,A)(N+\log_+1/\epsilon).$$
\end{theorem}

In contrast, note that the set $\mathcal Q=\{Q_1,Q_2\}$ where $Q_i$
is the  $p_i,q_i$ constant rate birth and death chain on $\{0,\dots,N\}$
and $p_1=q_2$, $q_1=p_2$ cannot be merging faster than $N^2$ because the
product $K=Q_1Q_2$ is, essentially, a simple random walk on a circle 
with almost uniform invariant measure. See \cite[Example 2.17]{SZ3}.

It may be illuminating to point out
that Theorem \ref{th-sing} is of some 
interest even in the time homogeneous case.
Suppose $K$ is irreducible aperiodic kernel with stationary measure $\pi$
and second largest singular value $\sigma$ on $\ell^2(\pi)$. Then we have
\begin{equation}
\left|\frac{K^n(x,y)}{\pi(y)}-1\right|\le [\pi(x)\pi(y)]^{-1/2}
\sigma^n.
\end{equation}
One difficulty attached to this estimate is that both
$[\pi(x)\pi(y)]^{-1/2}$ and
$\sigma$ depends on the perhaps  unknown stationary measure $\pi$.

Consider instead an initial measure $\mu_0>0$ and set $\mu_n=\mu_0K^n$.
Then we also have
\begin{equation}\label{hom1}
\left|\frac{K^n(x,y)}{\mu_n(y)}-1\right|\le [\mu_0(x)\mu_n(y)]^{-1/2}
\prod_1^n \sigma_i
\end{equation}
where $\sigma_i$ is the second largest singular value of $K: \ell^2(\mu_i)\ra \ell^2(\mu_{i-1})$. In particular, setting $\mu_0^*=\min_x\{\mu_0(x)\}$,
\begin{equation}\label{hom2}
\left|\frac{\pi(y)}{\mu_n(y)}-1\right|\le [\mu_0^*\mu_n(y)]^{-1/2}
\prod_1^n \sigma_i
\end{equation}
The estimates (\ref{hom1})-(\ref{hom2}) have the disadvantage that
each $\sigma_i$ depends on $\mu_0$ through $\mu_{i-1}$ and $\mu_i$. They have the
advantage that they do not depend in any direct way of $\pi$.
From a computational viewpoint, they offer a dynamical estimate of the error in 
the approximation of $\pi$ by $\mu_n$.

\subsection{An example where stability  fails}\label{stabfails}

In this section, we present a simple example that indicates why 
stability is a difficult property to study from a quantitative viewpoint.
Let $\Omega_N=\{0,1,\dots, N\}$, $N=2n+1$.
Fix $p,q,r\ge0$ with $p+q+r=1$, $p\neq q$, and $\eta_1\in [0,1)$.
Consider the Markov kernels $Q_1$ given by
\begin{eqnarray*}
Q_1(2x,2x+1)=p,&&  x=0,\dots,n\\
Q_1(2x,2x-1)=q, &&  x=1,\dots,n\\
Q_1(2x-1,2x)=q, &&  x=1,\dots,n\\
Q_1(2x+1,2x)=p, &&  x=0,\dots,n-1\\
Q_1(x,x)=r, &&  x=1,\dots, 2n,
\end{eqnarray*}
and
$$Q_1(0,0)=q+r, \;\; Q_1(N,N)=\eta_1, \;\;Q_1(N,N-1)=1-\eta_1.$$

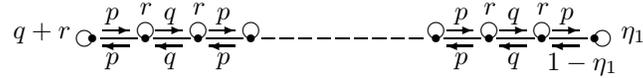
\begin{figure}[h]

\begin{center}
\caption{The chain with kernel $Q_1$} \vspace{.05in} \label{fig4-}
\end{center}

\begin{picture}(300, 30)(-50,0)
\put(30,15){\circle*{3}}\put(33,15){\line(1,0){14}}
\put(34,18){\vector(1,0){10}}\put(44,12){\vector(-1,0){10}}
\put(35,22){\makebox{$p$}}\put(35,5){\makebox{$p$}}
\put(27,15){\circle{6}}
\put(0,15){\makebox{$q+r$}}
\put(50,15){\circle*{3}}\put(53,15){\line(1,0){14}}
\put(54,18){\vector(1,0){10}}
\put(64,12){\vector(-1,0){10}}
\put(57,22){\makebox{$q$}}\put(57,5){\makebox{$q$}}
\put(50,18){\circle{6}}
\put(48,24){\makebox{$r$}}
\put(70,15){\circle*{3}}\put(73,15){\line(1,0){14}}
\put(74,18){\vector(1,0){10}}
\put(84,12){\vector(-1,0){10}}
\put(77,22){\makebox{$p$}}\put(77,5){\makebox{$p$}}
\put(70,18){\circle{6}}\put(68,24){\makebox{$r$}}
\put(90,15){\circle*{3}}\put(90,18){\circle{6}}
\multiput(94,15)(7,0){9}{\line(1,0){5}}

\put(220,15){\circle*{3}}\put(163,15){\line(1,0){14}}
\put(164,18){\vector(1,0){10}}
\put(174,12){\vector(-1,0){10}}
\put(167,22){\makebox{$p$}}\put(167,5){\makebox{$p$}}
\put(223,15){\circle{6}} \put(230,15){\makebox{$\eta_1$}}
\put(200,15){\circle*{3}}\put(203,15){\line(1,0){14}}
\put(204,18){\vector(1,0){10}}
\put(214,12){\vector(-1,0){10}}
\put(207,22){\makebox{$p$}}\put(202,3){\makebox{$1-\eta_1$}}
\put(200,18){\circle{6}} \put(198,24){\makebox{$r$}}
\put(180,15){\circle*{3}}\put(183,15){\line(1,0){14}}
\put(184,18){\vector(1,0){10}}
\put(194,12){\vector(-1,0){10}}
\put(187,22){\makebox{$q$}}\put(187,5){\makebox{$q$}}
\put(180,18){\circle{6}} \put(178,24){\makebox{$r$}}
\put(160,15){\circle*{3}}\put(160,18){\circle{6}}
\end{picture}
\end{figure}

This chain has reversible  measure  $\pi_1$ given by
$$\pi_1(0)=\dots=\pi_1(N-1)= (1-\eta_1)p^{-1}\pi_1(N)= \frac{(1-\eta_1)p^{-1}}{N(1-\eta_1)p^{-1}+1}.$$

Next, we let $Q_2$ be the kernel obtained by exchanging the roles of
$p$ and $q$ and replacing $\eta_1$ by $\eta_2\in [0,1)$. Obviously, this kernel has reversible measure
$\pi_2$ given by
$$\pi_2(0)=\dots=\pi_2(N-1)= (1-\eta_2)q^{-1}\pi_2(N)=
\frac{(1-\eta_2)q^{-1}}{N(1-\eta_2)q^{-1}+1}.$$

As long as $p,q$ are bounded away 
from $0$ and $1$  and $\eta_1,\eta_2$ are bounded away from $1$
these kernels $Q_1,Q_2$  
can be viewed as perturbations of the simple random walk on a
stick (with loops at the ends). Their respective invariant measures
are close to uniform. In fact, they are uniform if $\eta_1=q+r$, $\eta_2=p+r$. 

It is clear that, even if  
$r\eta_1\eta_2=0$, 
for any  sequence $(K_i)_1^\infty$  
with $K_i\in \{Q_1,Q_2\}$ we have
$$\min_{x,y\in \Omega_N}\{K_{m,m+2N+1}(x,y)\}\ge (\min\{p,q\})^{2N+1}>0.$$
Hence, if we let $\mu_0=u$ be the uniform measure and set $\mu_n=\mu_0K_{0,n}$
then there exists a constant $c=c(p,q,N)\in (1,\infty)$ such that
$$\forall\,n,\;\;\; c^{-1}\le \mu_n(x)\le c.$$
Further, it follows that any such sequence $(K_i)_1^\infty$ 
is merging in total variation and in relative-sup.

Nevertheless, we are going to show that the stability property 
fails at the quantitative level as $N$ tends to infinity. 
For this purpose, 
we compute the kernel of $K=Q_1Q_2$.  To understand $K$, it is useful
to imagine that the elements of $\{0,\dots,N\}$ arranged on a circle
with the even points in the upper half of the circle and the odd points
on the lower half of the circle. The only points on the horizontal
diameter of the circle are $0$ and $N$.

The kernel $K$ is given by the formulae:
\begin{eqnarray*}
K(2x,2x+2)= p^2,\; K(2x+2,2x)=q^2,&& x=0,\dots n-2,\\
K(2x+1,2x+3)=q^2,\; K(2x+3,2x+1)=p^2,&& x=0,\dots,n-2,\\
K(0,0)=2pq+r,\;\;K(x,x)= 2pq+r^2,  && x=1,\dots,N-2,\\
K(x,x+1)=K(x+1,x)= r(p+q)&& x=1,\dots,N-2,\\
K(0,1)=q^2+r(1-r),\; K(1,0)= p^2+ r(1-r), &&\\
K(N-1,N)=p\eta_2+r q,&&\\
K(N,N-1)=(1-\eta_2)\eta_1+ (1-\eta_1)r,&&\\
K(N-2,N)=q^2,\; K(N,N-2)=(1-\eta_1)p,&&\\
K(N-1,N-1)= p(q+1-\eta_2)+r^2,&&\\
K(N,N)=\eta_1\eta_2+ (1-\eta_1)q.&&
\end{eqnarray*}

The following special cases are of interest.
\begin{itemize}
\item[(i)] $r=0$, $\eta_1=q,\eta_2=p$. In this case $\pi_1=\pi_2$ is uniform
and $K$ is the kernel of a nearest-neighbors random walk on the circle with
transition probabilities $p^2$, $q^2$ and holding  $2pq$. Of course, this chain admits the uniform measure as invariant measure.
\item[(ii)] $r=0$, $\eta_1=\eta_2=0$. In this case, $K$ is essentially the kernel
of a $p'=p^2,q'=q^2,r'= 2pq$ birth and death chain. More precisely, after
writing $x_0=N, x_1=N-2,\dots, x_{n-1}=1, x_{n}=0, x_{n+1}=2, \dots, x_{N-1}= N-3, x_{N}=N-1$, we have
$$K(x_i,x_{i+1})=p^2,\;\;K(x_i,x_{i-1})=q^2, \;\;K(x_i,x_i)=2pq$$ except for
$K(x_0,x_1)=p$, $K(x_0,x_0)=q$, $K(x_N,x_N)=p+pq$.  This chain has invariant measure
$$\pi(x_i)=\pi(x_0)p^{-1}(p/q)^{2i}, \;i=1,\dots,N.$$
\end{itemize}
Using the same notation as in (ii) above,
we can compute the invariant measure $\pi$ of $K$ when $r=0$ for arbitrary values of $\eta_1,\eta_2$. Indeed,  $\pi$ must satisfy the following equations:
\begin{eqnarray*}
\pi(x_i)&=& 2pq \pi(x_i) +p^2 \pi(x_{i-1})+q^2\pi_i(x_{i+1}),\;\; i=2,\dots,N-1\\
\pi(x_1)&= &2pq \pi(x_1) +(1-\eta_1)p\pi(x_0)+ q^2\pi(x_2)\\
\pi(x_0)&=&(\eta_1\eta_2+(1-\eta_1)q)\pi(x_0)+ q^2\pi(x_1)+ p\eta_2 \pi(x_N)\\\pi(x_N)&=& p(q+1-\eta_2)\pi(x_N)+ (1-\eta_2)\eta_1 \pi(x_0)+  p^2 \pi(x_{N-1}).
\end{eqnarray*}
Because of the first equation, we set $\pi(x_i)= \alpha +\beta (p/q)^{2i}$
for $i=1,\dots, N$. This gives
\begin{eqnarray*}
(1-\eta_1)p\pi(x_0) &= &(\beta +\alpha)p^2\\
(p-\eta_1(\eta_2-q))\pi(x_0)&=&q^2(\alpha+\beta (p/q)^2)+ p\eta_2(\alpha+\beta(p/q)^{2N})\\
(1-\eta_2)\eta_1\pi(x_0)&=&  \alpha(q^2+p(\eta_2-p)) + p\eta_2\beta(p/q)^{2N}.
\end{eqnarray*}
Since the equations of the system $\pi=\pi K$ are not independent, the three 
equations above are not either. Indeed, subtracting  the last equation from the second yields the first. So the previous  system is equivalent to
\begin{eqnarray*}
(1-\eta_1)p^{-1}\pi(x_0) &= &\beta +\alpha\\
(1-\eta_2)\eta_1\pi(x_0)&=&  \alpha(q^2+p(\eta_2-p)) + p\eta_2\beta(p/q)^{2N}.
\end{eqnarray*}
Hence,  recalling that $q^2-p^2=q-p$ since $p+q=1$,
$$\beta= \frac{ (1-\eta_1)(q/p) -(1-\eta_2)}{q-p+p\eta_2(1-(p/q)^{2N})}\pi(x_0)$$
and
$$\alpha= \frac{(1-\eta_2)\eta_1- (1-\eta_1)\eta_2(p/q)^{2N}}{q-p +p\eta_2(1-(p/q)^{2N})}\pi(x_0).$$
When $\eta_1=\eta_2=0$ (resp. $\eta_1=q,\eta_2=p$), we recover $\alpha=0$, $\beta=p^{-1}\pi(x_0)$ (resp. $\alpha=\pi(x_0)$, $\beta=0$).

The denominator $q-p+p\eta_2(1-(p/q)^{2N})$ is positive or negative depending
on whether $q>p$ or $q<p$. By inspection of these formulae, 
one easily proves the following facts (the notation $x_i$ refers to the 
relabelling of the state space introduced in (ii) above).
\begin{itemize}
\item Assume that $q>p$, $r=0$. For any fixed $\eta_1>0$, there is a constant 
$c=c(p,q,\eta_1,\eta_2)\in (1,\infty)$ such that,  for  all large enough $N$, 
we have $$\forall\,x,\;\;c^{-1} \le (N+1)  \pi(x)\le c.$$
If $\eta_1=0$ then there is a constant 
$c=c(p,q,\eta_2)\in (1,\infty)$ such that,  for  all large enough $N$, 
we have $$\forall\,x_i,\;\;c^{-1} \le (q/p)^{2i}\pi(x_i)\le c.$$
\item Assume that $q<p$, $r=0$. For any fixed $\eta_2>0$, there is a constant 
$c=c(p,q,\eta_1,\eta_2)\in (1,\infty)$ such that,  for  all large enough $N$, 
we have $$\forall\,x,\;\;c^{-1} \le (N+1)  \pi(x)\le c.$$
If $\eta_2=0$ then there is a constant 
$c=c(p,q,\eta_1)\in (1,\infty)$ such that,  for  all large enough $N$, 
we have $$\forall\,x_i,\;\;c^{-1} \le (q/p)^{2(i-N)}\pi(x_i)\le c.$$
\end{itemize}

On the one hand, when
$r=\eta_1=\eta_2=0$ and $0<p\neq q<1$ are fixed, there are no 
constants $c$ independent  of $N$ for which  the 
set $\mathcal Q=\{ Q_1,Q_2\}$ is $c$-stable.  
One can even take $p_N,q_N$ so that 
$p_N/q_N =  1 +a N^{-\alpha} +o(N^{-1})$ as $N$ tends to infinity
with $a>0$ and $0<\alpha<1$. Then  $Q_1$ and $Q_2$ are asymptotically equal 
but there are no 
constants $c$ independent  of $N$ for which   
$\mathcal Q=\{ Q_1,Q_2\}$ is $c$-stable.  

On the other hand, when $0<p,q,r<1$, 
$\eta_1=q+r$ and  $\eta_2=p+r$, the uniform measure is
invariant for both kernels and $\mathcal Q$ is $1$-stable.

It seems likely that for fixed $\eta_1,\eta_2,r,p,q$ with  $0<p,q<1$
and  either $r>0$ or $\eta_1\eta_2>0$  the 
set $  \mathcal Q$ is $c$-stable but we do not know how to prove that.

\section{Time dependent edge weights}

In this section, we consider a family of graphs $\mathcal G_N=(\Omega_N,E_N)$.
These graphs are non-oriented with no multiple edges
(edges are pairs of vertices $e=\{x,y\}$ or singletons $e=\{x\}$).
We assume connectedness. 
We let $d(x)$ be the degree of $x$, i.e., $d(x)=\#\{e\in E: e\ni x\}$
and set
$$\delta(x)=\frac{d(x)}{\sum_x d(x)}.$$
For simplicity, we assume that
these graphs have bounded degree, i.e., 
$$\forall\,N,\;\;\forall\,x\in \Omega_N,\;\;d(x)\le D,$$ 
uniformly in $N$. 
A simple example
is the lazy stick of length $(N+1)$ as in Problem \ref{Pb0} and Figure \ref{fig2}.

\begin{figure}[h]

\begin{center}
\caption{The lazy stick} \vspace{.05in} \label{fig2}
\end{center}

\begin{picture}(300,10)(-50,0)
\put(30,15){\circle*{3}}\put(33,15){\line(1,0){14}}\put(27,15){\circle{6}}
\put(50,15){\circle*{3}}\put(53,15){\line(1,0){14}}\put(50,18){\circle{6}}
\put(70,15){\circle*{3}}\put(73,15){\line(1,0){14}}\put(70,18){\circle{6}}
\put(90,15){\circle*{3}}\put(90,18){\circle{6}}
\multiput(94,15)(7,0){9}{\line(1,0){5}}
\put(220,15){\circle*{3}}\put(163,15){\line(1,0){14}}
\put(223,15){\circle{6}}
\put(200,15){\circle*{3}}\put(203,15){\line(1,0){14}}\put(200,18){\circle{6}}
\put(180,15){\circle*{3}}\put(183,15){\line(1,0){14}}\put(180,18){\circle{6}}
\put(160,15){\circle*{3}}\put(160,18){\circle{6}}
\end{picture}
\end{figure}
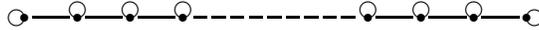

\subsection{Adapted kernels}

For any choice of positive weights
$\mathbf w=(w_e)_{e\in E}$ on $\mathcal G_n$, we obtain a reversible
Markov kernel
$K(\mathbf w)$ with support on pairs $(x,y)$ such that $\{x,y\}\in E$, in which case
$$K(\mathbf w)(x,y)= \frac{w_{\{x,y\}}}{\sum_{e\ni x}w_e}.$$
The reversible measure is
$$\pi(\mathbf w)(x)=
c(\mathbf w)^{-1}\sum_{e\ni x}w_e, \;\;c(\mathbf w)= \sum_x\sum_{e\ni x}w_e.$$
For instance, picking $\mathbf w=\mathbf 1$, i.e., $w_e=1$ for all $e\in E$,
we obtain the kernel
$K_{\mbox{\tiny sr}}(x,y)=K(\mathbf 1)(x,y)=1_E(\{x,y\})/d(x)$
of the simple random walk on the given graph. The reversible measure for
$K_{\mbox{\tiny sr}}$ is $\pi(\mathbf 1)=\delta$.

Set $$R(\mathbf w)=\max\left\{w_e/w_{e'}: e,e'\in E\right\}.$$
Observe that $R(\mathbf w)\le b$ implies
\begin{equation}
\forall\,x,\;\;  b^{-1}\delta(x)\le \pi(\mathbf w)(x)\le b\delta(x).
\end{equation}
For instance, to prove the upper bound, let $w_0=\min\{w_e\}$ and write
$$\pi(\mathbf w)(x)=
c(\mathbf w)^{-1}\sum_{e\ni x}w_e\le \frac{1}{\sum_x d(x)}
\sum_{e\ni x}\frac{w_e}{w_0}\le b\delta(x).$$
The proof of the lower bound is similar. Further, we also have
\begin{equation}
\forall\,x,y,\;\;  (Db)^{-1}\pi(\mathbf w)(y)\le 
\pi(\mathbf w)(x)\le Db \pi(\mathbf w)(y).
\end{equation}
Indeed, $\sum_{e\ni x}w_e \le Db  w_0\le Db\sum_{e\ni y}w_e$. 

For any $N$ and $b>1$, set
$$\mathcal Q(\mathcal G_N,b)=\{K(\mathbf w): R(\mathbf w)\le b\}.$$
For any $N$, $b>1$ and fixed probability measure $\pi$ on $\Omega_N$, set
$$\mathcal Q(\mathcal G_N,b,\pi)
=\{K(\mathbf w): R(\mathbf w)\le b,\;\pi(\mathbf w)=\pi\}.$$
The set of weight $\mathcal Q(\mathcal G_N,b,\pi)$
may well be empty. However,
we can use the Metropolis
algorithm construction  to prove the following lemma.
\begin{lemma} \label{lem-Met}
Assume that $\{x\}\in E$  for all $x$
(i.e, the graphs $\mathcal G_N$ have a loop at each vertex)
and  that $a^{-1}\le \pi(x)/\delta(x)\le a$. Then the set
$\mathcal Q(\mathcal G_N,a^2(b^3+bD),\pi)$ is non-empty
for any $b\ge 1$. It contains a continuum
of kernels $K(\mathbf w)$ for any $b>1$.
\end{lemma}
\begin{proof}
Starting from any weight $\mathbf v$ with $R(\mathbf v)\le b$,
we define a new weight $\mathbf w$ by setting
$$\forall\, \{x,y\}\in E, x\neq y,\;\;w_{\{x,y\}}=
v_{\{x,y\}}
\min\left\{
\frac{\pi(x)}{\pi(\mathbf v)(x)}, \frac{\pi(y)}{\pi(\mathbf v)(y)}\right\}$$
and
$$w_{\{x\}}= c(\mathbf v)\pi(x)
-\sum_{y\neq x}
v_{\{x,y\}}
\min\left\{ \frac{\pi(x)}{\pi(\mathbf v)(x)}, \frac{\pi(y)}{\pi(\mathbf v)(y)}\right\}.$$
It is clear that $\pi(\mathbf w)=\pi$ (Indeed, $K(\mathbf w)$ is the 
kernel of the Metropolis algorithm chain
for $\pi$ with proposal based on $K(\mathbf v)$).
Further, since
$$\sum_{y\neq x}v_{\{x,y\}}
\min\left\{ \frac{\pi(x)}{\pi(\mathbf v)(x)}, \frac{\pi(y)}{\pi(\mathbf v)(y)}\right\}\le \pi(x)\left(c(\mathbf v)-\frac{v_{\{x\}}}{\pi(\mathbf v)(x)
}\right),$$
we have
$$ \frac{\pi(x)v_{\{x\}}}{\pi(\mathbf v)(x)}\le    w_{\{x\}} \le c(\mathbf v)\pi(x).$$
Now, since $a^{-1}\delta(x)\le \pi(x)\le a\delta(x)$ and
$\mathbf v \in \mathcal Q(\mathcal G_N,b)$, we obtain
$$\forall\,x\neq y, x'\neq y',\;\;\frac{w_{\{x,y\}}}{w_{\{x',y'\}}}\le b^3a^2.$$
and
$$\forall\,\{x,y\}\in E, x',\;\;
\max\left\{\frac{w_{\{x,y\}}}{w_{\{x'\}}},\frac{w_{\{x'\}}}{w_{\{x,y\}}}
\right\}\le a^2bD.$$
Hence $R(\mathbf w)\le a^2(b^3+bD)$ and  
$K(\mathbf w)\in \mathcal Q(\mathcal G_N,a^2(b^3+bD), \pi)$ as desired.
\end{proof}

\subsection{Time homogeneous results}
For each $N$, let $\sigma_N$ be the second singular value
of $(K_{\mbox{\tiny sr}},\delta)$, i.e., the second largest eigenvalue in
absolute value of the simple random walk on $\mathcal G_N$. For instance,
for the ``lazy stick'' of Figure \ref{fig2}, $1-\sigma_N $ is of order $1/N^2$.
For any $\mathbf w$, let $\sigma(\mathbf w)$ be the second largest
singular value of $(K(\mathbf w), \pi(\mathbf w))$. The following lemma
concerns the time homogeneous chains associated with kernels in 
$\mathcal Q(\mathcal G_N,b)$.

\begin{proposition} \label{lem-thsv}
For any $b\ge 1$ and any $K(\mathbf w)\in \mathcal Q(\mathcal G_N,b)$
$$   b^{-2}(1-\sigma_N)\le 1-\sigma(\mathbf w)  .$$
In particular, uniformly over $w\in \mathcal Q(\mathcal G_N,b)$,
\begin{equation}\label{cvth1}
\left|\frac{K(\mathbf w)^n(x,y)}{\pi(\mathbf w)(y)}-1\right|\le
bd_*^{-1}\Delta_N(1- b^{-2}(1-\sigma_N))^n,
\end{equation}
with $ \Delta_N=\sum_x d(x)$,
$d_*=\min_x \{d(x)\}.$
\end{proposition}
\begin{proof} This is based on the basic comparison techniques
of \cite{DS-C}. In the  present case, it is best to compare the lowest
and second largest eigenvalues of $K_{\mbox{\tiny sr}}$,  
call them $\beta_-$ and $\beta_1$, respectively, with the same quantities
$\beta_-(\mathbf w)$ and $\beta_1(\mathbf w)$ relative to
$K(\mathbf w)$. The relation with the singular value 
$\sigma(\mathbf w)$ is given by
$\sigma(\mathbf w)=\max \{ -\beta_-(\mathbf w),
\beta_1(\mathbf w)\}$. For comparison purpose, one uses the Dirichlet forms
(recall that edges here are (non-oriented) pairs $\{x,y\}$)
$$\mathcal E_{\mathbf w}(f,f)=\frac{1}{c(\mathbf w)}
\sum_{e=\{x,y\}}|f(x)-f(y)|^2w_e$$
and $$\mathcal E_{\mbox{\tiny sr}}(f,f)=
\mathcal E_{\mathbf 1}(f,f)=\frac{1}{\Delta_N}\sum_{e=\{x,y\}}|f(x)-f(y)|^2.$$
Clearly, for any $f$,
\begin{equation}\label{comp1}
\mathcal E_{\mbox{\tiny sr}}(f,f)\le \frac{c(\mathbf w)b}{\Delta_N}\mathcal E_{\mathbf w}(f,f),\;\;\mbox{Var}_{\pi(\mathbf w)}(f)\le \frac{\Delta_Nb}{c(\mathbf w)}
\mbox{Var}_\delta(f).\end{equation}
This yields $1-\beta_1\le b^2(1-\beta_1(\mathbf w))$.  A similar argument using
(the sum here is over all $x,y$ with $\{x,y\}\in E$,
which explains the $\frac{1}{2}$ factor)
$$\mathcal F_{\mathbf w}(f,f)=
\frac{1}{2c(\mathbf w)}\sum_{x,y:\{x,y\}\in E}|f(x)+f(y)|^2w_{\{x,y\}}$$
yields $1+\beta_-\le b^2(1+\beta_-(\mathbf w))$. This gives
the desired result.
\end{proof}

\begin{example} For our present purpose, call
``$(d,\epsilon)$-expander family''
any infinite family of regular  graphs $\mathcal G_N$ of fixed degree $d$,
with $|\Omega_N|=\# \Omega_N$ tending to infinity with $N$ and
satisfying $\sigma_N\le 1-\epsilon$. See \cite{HLW,Lub2}
for various related definitions and discussions of particular examples.
Proposition \ref{lem-thsv} shows that for any $K(\mathbf w)\in
\mathcal Q(\mathcal G_N,b)$, we have
$$\left|\frac{K(\mathbf w)^n(x,y)}{\pi(\mathbf w)(y)}-1\right|\le
b|\Omega_N|(1- \epsilon/b^{2})^n,$$
\end{example}

Let us point out that, beside singular values ,  there are further related techniques 
that yield complementary results. They  include
the use of Nash and logarithmic Sobolev inequalities (modified or not).
See \cite{DS-N,DS-L,SZ,SZ4}.
For instance, to show that on the ``lazy stick'' $\mathcal G_N$ of
Figure \ref{fig2}, any chains with kernel in $\mathcal Q(\mathcal G_N,b)$
converges to stationarity in order $N^2$, one uses the Nash inequality technique of \cite{DS-N}.

\subsection{Time inhomogeneous chains}

A fundamental question about time inhomogeneous Markov chains is 
whether or not a result similar to  (\ref{cvth1}) holds true for time
inhomogeneous chains with kernels in $\mathcal Q_N(\mathcal G_N,b)$.
Little is known about this.

Fix $b> 1$. Let $(K_i)_1^\infty$ be a sequence of Markov kernels in
$\mathcal Q(\mathcal G_N,b)$ and $K_{m,n}$ be the associated iterated kernel.
Recall  that the property ``$\sigma_N<1$'' 
is equivalent to the irreducibility and
aperiodicity of $K_{\mbox{\tiny sr}}$. Because all the kernels in
$\mathcal Q(\mathcal G_N,b)$ are (uniformly) adapted to the graph
structure $\mathcal G_N$, there exists $ \ell=\ell(N,b)$ and
$\epsilon=\epsilon(N,b)>0$ such that, for all n,
$K_{n,n+\ell}(x,y)\ge \epsilon$. As explained in Section \ref{FSS},
this implies relative-sup merging for any such time inhomogeneous chain.
However, this result is purely qualitative. No acceptable
quantitative result can be obtain by such an argument.

\begin{problem} \label{wpb1}
Fix reals $D,b>1$. Prove or disprove
that there exists a constant $A$ such that
for any family $\mathcal G_N$ with maximal degree at most $D$,
any sequence $(K_i)_1^\infty$ with  $K_i\in \mathcal Q(\mathcal G_N,b)$, 
any initial
distributions $\mu_0,\mu_0'$ and any $\epsilon>0$, if
$$n\ge A (1-\sigma_N)^{-1}(\log|\Omega_N|+\log_+(1/\epsilon))$$
then  $\mu_n=\mu_0K_{0,n}$ and $\mu'_n=\mu_0'K_{0,n}$ satisfy
$$\max_{x\in \Omega_N} \left\{
\left|\frac{\mu'_n(x)}{\mu_n(x)}-1\right|\right\}\le \epsilon.$$
\end{problem}
This is an open problem, even for the ``lazy stick'' of Figure \ref{fig2}.
It seems rather unclear whether one should except a positive answer or not.

Next, we consider  another question, quite interesting but, a priori,
of a different nature. Recall that, given $\mathcal G_N$,
 $\delta$  denotes  the normalized
reversible measure of $K_{\mbox{\tiny sr}}$.

\begin{problem} \label{wpb2}
Fix reals $D,b>1$. Prove or disprove
that there exists a constant $A\ge 1$ such that
for any family $\mathcal G_N$ with maximal degree at most $D$,
any sequence $(K_i)_1^\infty$with  $K_i\in \mathcal Q(\mathcal G_N,b)$
and any initial distributions $\mu_0$, if
$$n\ge A (1-\sigma_N)^{-1}(\log|\Omega_N|)$$
then  $\mu_n=\mu_0K_{0,n}$ satisfies
$$\forall\,x\in \Omega_N,\;\;\; A^{-1}\le \frac{\mu_n(x)}{\delta(x)}
\le A.$$
\end{problem}

In words, a positive solution to Problem \ref{wpb1}
yields the relative-sup  merging in time of order
at most $A(1-\sigma_N)^{-1}\log|\Omega_N|$, uniformly
for any time inhomogeneous chain with kernels in
$\mathcal Q(\mathcal G_N,b)$ whereas a positive solution to Problem \ref{wpb2} would indicate
that, after a time of order at most
$A(1-\sigma_N)^{-1}\log|\Omega_N|$, uniformly
for any time inhomogeneous chain with kernels in
$\mathcal Q(\mathcal G_N,b)$ and for any initial distribution $\mu_0$,
the measure $\mu_n=\mu_0K_{0,n}$ is comparable to $\delta$.
In fact, because of the uniform way in which Problem \ref{wpb2} is formulated,
a positive answer implies
that the measure $\delta$ is $A$-stable for $\mathcal Q(\mathcal G_N,b)$.

At this writing, the best evidence for a positive answer to these problems
is contained in the following two partial results.
The first result concerns sequences whose kernels share the same
invariant distribution. For the proof, see \cite{SZ}.\
\begin{theorem} \label{th-w1}
Fix reals $D,b>1$ and  measures $\pi_N$
on $\Omega_N$.  Assume that  $\mathcal G_N$ has maximal degree at most $D$
and that $\mathcal Q(\mathcal G_N,b,\pi_N)$ is non-empty.
Under these circunstances, there is a constant $A=A(D,b)$ such that
for any $\epsilon>0$, any sequence $(K_i)_1^\infty$ with
$K_i\in \mathcal Q(\mathcal G_N,b,\pi_N)$ and any pair $\mu_0, \mu'_0$
of initial distributions, if
$$n\ge A (1-\sigma_N)^{-1}(\log|\Omega_n|+\log_+(1/\epsilon))$$
then  $\mu_n=\mu_0K_{0,n}$ and $\mu'_n=\mu_0'K_{0,n}$ satisfy
$$\max_{x\in \Omega_N} \left\{
\left|\frac{\mu'_n(x)}{\mu_n(x)}-1\right|\right\}\le \epsilon.$$
\end{theorem}
Note the the hypothesis that $\mathcal Q(\mathcal G_N,b,\pi_N)$ is non-empty
implies that $b^{-1}\le \pi_N/\delta\le b$.
The  second result assumes $c$-stability. For the proof, see \cite{SZ4}.
\begin{theorem} \label{th-w2}
Fix reals $D,b,c>1$. Assume that $\mathcal G_N$ has maximal degree at most $D$.
Let $(K_i)_1^\infty$ be a sequence of kernels on $\Omega_N$ with
$K_i\in \mathcal Q(\mathcal G_N,b)$. Assume that the distribution
$\delta$ on $\Omega_N$ is $c$-stable for $(K_i)_1^\infty$. Then there exists
a constant $A=A(D,b,c)$ such that
for any $\epsilon>0$ and pair $\mu_0, \mu'_0$
of initial distributions, if
$$n\ge A (1-\sigma_N)^{-1}(\log|\Omega_n|+\log_+(1/\epsilon))$$
then  $\mu_n=\mu_0K_{0,n}$ and $\mu'_n=\mu_0'K_{0,n}$ satisfy
$$\max_{x\in \Omega_N} \left\{
\left|\frac{\mu'_n(x)}{\mu_n(x)}-1\right|\right\}\le \epsilon.$$
\end{theorem}

Theorem \ref{th-w1} can be viewed as a special case of Theorem \ref{th-w2}.
Indeed, if $\mathcal Q(\mathcal G_N,b,\pi_N)$ is not empty then we must have
$b^{-1}\delta\le \pi_N\le b^1 \delta$ so that $\delta$ is a $b$-stable
measure for any sequence of kernels in $\mathcal Q(\mathcal G_N,b,\pi_N)$.
By Lemma \ref{lem-Met}, it is not difficult to produce examples where
Theorem \ref{th-w1} applies.  Finding examples of application of Theorem
\ref{th-w2} (where the $K_i$'s  do not all share the same invariant
distribution) is a difficult problem. 

Under the stability hypothesis of Theorem \ref{th-w2}, methods such as Nash
inequalities and logarithmic Sobolev inequality can also be applied.
See \cite{SZ4}.

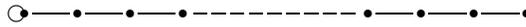
\begin{figure}[h]
\begin{center}
\caption{The underlying graph for the kernels $Q_1,Q_2$ of Section 
\ref{stabfails}} 
\vspace{.05in} \label{figfail}
\end{center}

\begin{picture}(300,10)(-50,0)
\put(30,15){\circle*{3}}\put(33,15){\line(1,0){14}}\put(27,15){\circle{6}}
\put(50,15){\circle*{3}}\put(53,15){\line(1,0){14}}
\put(70,15){\circle*{3}}\put(73,15){\line(1,0){14}}
\put(90,15){\circle*{3}}
\multiput(94,15)(7,0){9}{\line(1,0){5}}
\put(220,15){\circle*{3}}\put(163,15){\line(1,0){14}}
\put(200,15){\circle*{3}}\put(203,15){\line(1,0){14}}
\put(180,15){\circle*{3}}\put(183,15){\line(1,0){14}}
\put(160,15){\circle*{3}}
\end{picture}
\end{figure}

\begin{remark} Consider the kernels $Q_1,Q_2 $ of Section \ref{stabfails}, 
with fixed $p,q,r,\eta_1,\eta_2$ with $r=\eta_1=\eta_2=0$ and $0<p\neq q<1$.
The kernels $Q_1,Q_2$ are adapted to the graph structure of Figure 
\ref{figfail}. We proved in Section \ref{stabfails} that 
stability fails for $\mathcal Q=\{Q_1,Q_2\}$. Even on the
``lazy stick'' of Figure \ref{fig2}, we do not understand whether stability
holds or not. An interesting example of stability on the lazy stick
is proved   in \cite{SZ3}. This example involves perturbations that are
localized at the ends of the stick. Further examples are discussed in
\cite{SZwave}.
\end{remark}

\end{document}